\documentclass[12pt,a4paper]{article}
\usepackage[T1]{fontenc}
\usepackage[utf8]{inputenc}
\usepackage[british]{babel}
\usepackage{amsmath,amssymb,amsthm}
\usepackage{wasysym}
\usepackage{tensor}
\usepackage{xcolor}
\usepackage[all]{xy}
\usepackage{cite}

\usepackage[shortlabels]{enumitem}
\setlist[enumerate]{leftmargin=*}

\usepackage{hyperref}

\newcommand{\gen}[1]{\left\langle #1 \right\rangle}
\makeatletter
\newcommand{\leqnomode}{\tagsleft@true\let\veqno\@@leqno}
\newcommand{\reqnomode}{\tagsleft@false\let\veqno\@@eqno}
\makeatother

\DeclareMathOperator\End{End}
\newcommand\tsba[1]{{#1}_{{\oplus}, {\odot}}}
\newcommand\ZLV{ZLV}

\def\Z{\mathbb{Z}}

\newcommand\C{\mathbb{C}} 
  
\newcommand\F{\mathbb{F}}

\newcommand\x{\times}

\newcommand\Aut{\operatorname{Aut}}
\newcommand\GL{\operatorname{GL}}
\newcommand\PGL{\operatorname{PGL}}
\newcommand\Core{\operatorname{Core}}

\newcommand\fL{\mathsf{L}} 

\newtheorem{teo}{Theorem}[section]
\newtheorem{prop}[teo]{Proposition}

\theoremstyle{definition}
\newtheorem{defi}[teo]{Definition}

\theoremstyle{remark}
\newtheorem{remark}[teo]{Remark}

\theoremstyle{definition}
\newtheorem{ex}[teo]{Example}

\definecolor{White}{RGB}{250,250,250}
\numberwithin{equation}{section}
\title{Representations of finite skew braces}
\author{A. Ballester-Bolinches \and R. Esteban-Romero \and P. P{\'e}rez-Altarriba\thanks{Departament de Matem\`atiques, Universitat de Val\`encia; Dr.\ Moliner, 50; 46100 Burjassot, Val\`encia, Spain; \href{mailto:Adolfo.Ballester@uv.es}{\texttt{Adolfo.Ballester@uv.es}}, \href{mailto:Ramon.Esteban@uv.es}{\texttt{Ramon.Esteban@uv.es}}, \href{mailto:Pedro.A.Perez@uv.es}{\texttt{Pedro.A.Perez@uv.es}}; ORCID 0000-0002-2051-9075, 0000-0002-2321-8139, 0009-0009-7082-9002}}
\begin{document}
\maketitle
\begin{abstract}
  One of the classical open problems in the theory of skew left braces is the study of their representation theory. We propose in this paper a definition of representation of a skew left brace and study its properties. Representations of the trifactorised groups associated with skew left braces play a fundamental role. 

  \emph{Mathematics Subject Classification (2020):}
  16T25, % Yang-Baxter equations 
  81R50, % Quantum groups and related algebraic methods applied to problems in quantum theory
  20C35, % Application of group representations to physics and other areas of science
  20C99, % None of the above, but in this section (20Cxx: Representation theory of groups)
  20D40.  % Products of subgroups of abstract finite groups

\end{abstract}

\section{Introduction}
The concept of \emph{skew left brace} was introduced by Guarnieri and Vendramin \cite{GuarnieriVendramin17} as a triple $(B, {+}, {\cdot})$, where $(B, {+})$ and $(B, {\cdot})$ are groups linked by the modified distributive law $a(b+c)=ab-a+ac$ for all $a$, $b$, $c\in B$. This notion generalises the \emph{left braces} introduced by Rump \cite{Rump07} as a generalisation of Jacobson radical rings, corresponding to the case in which $(B, {+})$ is abelian, although this structure seems to be already known by Kurosh (see \cite[Chapter~10]{Kurosh74}). Skew left braces constitute the precise framework to study non-degenerate bijective set-theoretic solutions of the Yang-Baxter equation, a fundamental equation of the mathematical physics introduced independently by Yang \cite{Yang67} and Baxter \cite{Baxter73}.

Representation theory is the branch of mathematics that studies algebraic structures by expressing their elements as linear transformations of vector spaces, corresponding with matrices in the case of vector spaces of finite dimension, and by identifying the operations in the algebraic structure with suitable operations with the linear transformations or the corresponding matrices. In the case of representations of groups, the elements of the group are expressed by invertible linear transformations or regular matrices and the operation is identified with the composition of linear transformations or the product of matrices. A group representation of a group $G$ with representation space~$V$ is then defined as a group homomorphism $\rho\colon G\longrightarrow \GL(V)$. Since linear algebra and matrix groups are quite well understood, representation theory of groups has become a very powerful tool in group theory.

One of the open problems in the theory of skew left braces is precisely the study of representations of skew left braces (see, for instance, \cite[Problem~2.26]{Vendramin19-agta}). As both operations of a skew left brace correspond to  group operations, it seems natural to encode the elements of the skew left brace by means of invertible linear maps or regular matrices. However, \emph{a priori}, there is no canonical skew left brace structure on a group of regular matrices, where we only have one operation, namely, the product of matrices. It seems natural to define a representation of a skew left brace~$B$ with representation space~$V$ as a map $\rho\colon B\longrightarrow \GL(V)$ for which the image $\rho(B)$ acquires a structure of skew left brace, that must be determined, and the restriction of~$\rho$ to its image becomes a homomorphism of skew left braces. In the case of a group representation $\rho\colon G\longrightarrow \GL(V)$, since $\rho$ is a group homomorphism, we have that its kernel is a normal subgroup of~$G$. Analogously, since the kernel of a homomorphism between two skew left braces is an ideal of the first one, we must expect the kernel of a representation of a skew left brace~$B$ to be an ideal of~$B$. This motivates the proposal we present in Section~\ref{sec-Br-rep}.

It is possible to associate to skew left braces some types of trifactorised groups, as shown in \cite{Sysak11-Ischia10,BallesterEsteban22,BallesterEstebanPerezAPerezC25-arXiv-categoriesskewleftbraces}, and to describe substructures and morphisms of skew left braces in terms of suitable substructures and morphisms of the corresponding trifactorised groups. As trifactorised group morphisms are just group homomorphisms satisfying some additional conditions, we will define in Section~\ref{sec-3fact-rep} a trifactorised group representation of a trifactorised group $(G, K, H, E)$ with representation space $V$ as a group homomorphism $\rho\colon G\longrightarrow \rho(V)$ for which the kernel corresponds to an ideal of the associated skew left brace. We will see in Section~\ref{sec-ind-3fact-rep} how to pass from a representation of a finite skew left brace a representation of the corresponding trifactorised group.

There have been other attempts in the mathematical literature to define representations of skew left braces. We will pay attention to the definition of skew brace representation given by Zhu in \cite[Remark~3.3]{Zhu22} for skew left braces with abelian additive group and extended in the natural way to skew left braces by Letourmy and Vendramin in~\cite[Definition~4.1]{LetourmyVendramin24}. In this setting, a representation of a skew left brace $B$ with representation space $V$, called here a \emph{\ZLV-representation} to avoid confusions, consists a pair of group homomorphisms $(B, {+})\longrightarrow \GL(V)$ and $(B, {\cdot})\longrightarrow \GL(V)$ satisfying a compatibility condition. We analyse this definition in Section~\ref{sec-ZLV-rep} and compare it with our definitions. We note that the kernel of the homomorphism corresponding to the additive group of a \ZLV-representation is not, in general, an ideal of the skew left brace. We will conclude in Proposition~\ref{prop-ZLV-large} that the \ZLV-representations of a skew left brace~$B$ are representations of the semidirect product $G=[(B, {+})](B, {\cdot})$ of the additive group and the multiplicative group of $B$ via the lambda action.

\section{Representations of skew left braces}\label{sec-Br-rep}
Let $(B, {+}, {\cdot})$ be a skew left brace with additive group $K=(B, {+})$ and multiplicative group $C=(B, {\cdot})$. We recall (see, for instance, \cite{BallesterEsteban22}) that for every $a\in B$, we can define $\lambda_a\in \Aut(K)$ by $\lambda_a(b)=-a+ab$ and so $\lambda\colon C\longrightarrow \Aut(K)$ given by $\lambda(c)=\lambda_c$ for every $c\in C$ is a group homomorphism, called the \emph{lambda action} or \emph{lambda map}. Moreover, the identity map $\delta\colon C\longrightarrow K$ is a derivation or $1$-cocycle with respect to~$\lambda$. We note that a skew left brace is univocally determined by one of the operations and the lambda action. 

A \emph{subbrace} of $B$ is a subgroup of $K$ which is also a subgroup of $C$. A subbrace $L$ of $B$ is said to be a \emph{strong left ideal} (respectively \emph{right ideal}) of $B$ if $L$ is a $\lambda$-invariant normal subgroup of $K$. An \emph{ideal} of $B$ is a strong left ideal which is a normal subgroup of $C$.

To approach the notion of representation of a skew left brace, we begin with a group representation $\rho\colon K\longrightarrow \GL(V)$ of the additive group~$K$. Since we want to define a skew left brace structure on $\rho(K)$, the first thing we need is to find a necessary and sufficient condition to have a lambda action on $\rho(K)$ compatible with the lambda action of $B$.

\begin{prop}\label{prop-br-semirep-equi}
  Let $B$ be a skew left brace with additive group $K$, multiplicative group~$C$, and lambda action $\lambda$. Let $\rho\colon K\longrightarrow\GL(V)$ be a group representation. Then the following statements are equivalent:
\begin{enumerate}
\item There exists a map $\beta\colon C\longrightarrow\Aut(\rho(K))$ such that $\rho\circ\lambda_c=\beta_c\circ\rho$ for all $c\in C$, where $\beta_c=\beta(c)$ for all $c\in C$.\label{en-semirep-equi-1}
\item $\ker\rho$ is a strong left ideal of~$B$.\label{en-semirep-equi-2}
\end{enumerate}
In this case, the map $\beta$ is a group homomorphism and it is uniquely determined by~$\rho$.
\end{prop}
\begin{proof}
\emph{\ref{en-semirep-equi-1} implies~\ref{en-semirep-equi-2}}. Suppose that there exists $\beta\colon C\longrightarrow\Aut(\rho(K))$ such that $\rho\circ\lambda_c=\beta_c\circ\rho$ for all $c\in C$. Since $\ker\rho$ is a normal subgroup of $K$, it is enough to check that $\ker\rho$ is invariant under the action of $\lambda_c$ for all $c\in C$. Let $k\in \ker\rho$. Then $\rho(\lambda_c(k))=\beta_c(\rho(k))=\beta_c(1)=1$. This implies that $\lambda_c(k)\in\ker\rho$. 

\emph{\ref{en-semirep-equi-2} implies~\ref{en-semirep-equi-1}}. Conversely, suppose that $\ker\rho$ is a strong left ideal of~$B$. Let $c\in C$. If $\rho(k)=\rho(k')$, then $k'=k+x$ with $x\in \ker\rho$. Since $\ker\rho$ is a strong left ideal of~$B$, $\lambda_c(x)\in \ker\rho$ and $\rho(\lambda_c(x))=1$. Hence, $\rho(\lambda_c(k'))=\rho(\lambda_c(k+x))=\rho(\lambda_c(k))\rho(\lambda_c(x))=\rho(\lambda_c(k))$. Consequently, we can define $\beta_c\colon \rho(K)\longrightarrow\rho(K)$ by means of $\beta_c(\rho(k))=\rho(\lambda_c(k))$.

Given $c$, $d\in C$ and $k\in K$, $\beta_{cd}(\rho(k))=\rho(\lambda_{cd}(k))=\rho(\lambda_c(\lambda_d(k)))=\beta_c(\rho(\lambda_d(k))=\beta_c(\beta_d(\rho(k)))$. It follows that $\beta_{cd}=\beta_c\circ \beta_d$ and so $\beta\colon C\longrightarrow\Aut(\rho(K))$ is a group homomorphism. The unicity follows directly from the condition $\rho\circ\lambda_c=\beta_c\circ\rho$ for all $c\in C$.
\end{proof}

\begin{defi}\label{defi-br-semi-rep}
Consider a group representation $\rho\colon K\longrightarrow\GL(V)$. We say that $\rho$ is a \emph{brace semi-representation} of $B$ when $\ker\rho$ is a strong left ideal of~$B$. The group homomorphism $\beta\colon C\longrightarrow\Aut(\rho(K))$ defined in Proposition~\ref{prop-br-semirep-equi} will be called the \emph{beta map} of the semi-representation and the vector space~$V$ will be called the \emph{representation space} of~$\rho$.
\end{defi}

Not every group representation is brace semi-representation.
\begin{ex}\label{ex-grp-rep-not-br-semi-rep}
  Let us consider the skew left brace $B$ with additive group $K=\langle a, b \mid 2a=2b=0, a+b=b+a\rangle$, multiplicative group $C=\langle c\mid c^4=1\rangle$, lambda action given by $\lambda_c(a)=a+b$, $\lambda_c(b)=b$, and bijective derivation $\delta\colon C\longrightarrow K$ given by $\delta(1)=0$, $\delta(c)=a$, $\delta(c^2)=b$, $\delta(c^3)=a+b$. We can consider the group representation $\rho\colon K\longrightarrow\C$ given by $\rho(a)=1$, $\rho(b)=-1$. It follows that $\ker\rho=\langle a\rangle$, that is not a left ideal because it is not lambda-invariant. Therefore, $\rho$ is not a brace semi-representation.
\end{ex}

Let $\rho\colon K\longrightarrow\GL(V)$ be a brace semi-representation of~$B$. In $\rho(K)$, we have an operation and the beta map, which bears a resemblance to a lambda map, but, in general, we do not have a skew left brace structure. The natural way to define a second operation in $\rho(K)$ would be
\begin{equation}\label{eq-prod-br-rep}
\rho(\delta(c))\boxdot\rho(\delta(c'))=\rho(\delta(c))\beta_c(\rho(\delta(c'))),
\end{equation}
where $\delta\colon C\longrightarrow K$ is the identity map and $\beta$ is the beta map. Unfortunately, there could be $c$, $\hat c\in C$ such that $\rho(\delta(c))=\rho(\delta(\hat c))$, but $\beta_c\ne \beta_{\hat c}$. To solve this problem we need to add another condition.

\begin{prop}\label{prop-br-rep-equi}
Consider a brace semi-representation $\rho\colon K\longrightarrow\GL(V)$ of~$B$ with beta map $\beta\colon C\longrightarrow\Aut(\rho(K))$. Then the following statements are equivalent:
\begin{enumerate}
\item $\ker\rho\subseteq\delta(\ker\beta)$.\label{en-rep-id-1}
\item If $c$, $\hat c\in C$ and $\rho(c)=\rho(\hat c)$, then $\beta_c=\beta_{\hat c}$.\label{en-rep-id-2}  
\item $\ker\rho$ is an ideal of~$B$.\label{en-rep-id-3}
\end{enumerate}

If one of these statements holds, then $(\rho(K),{\cdot}, {\boxdot})$ is a skew left brace isomorphic to $B/{\ker\rho}$ and its lambda action $\hat\beta\colon (\rho(K),\boxdot)\longrightarrow\Aut(\rho(K))$ is given by $\hat\beta(\rho(\delta(c)))=\beta_c$.
\end{prop}
\begin{proof}
  \emph{\ref{en-rep-id-1} implies~\ref{en-rep-id-2}}. Note that, for $c$, $\hat c\in C$, we have that
  \[\delta(c^{-1}\hat c)=\delta(c^{-1})+\lambda_{c^{-1}}(\delta(\hat c)).\]
  When we apply this formula to $\hat c=c$, we obtain that
  \[0=\delta(1)=\delta(c^{-1}c)=\delta(c^{-1})+\lambda_{c^{-1}}(\delta(c))\]
  and so $\delta(c^{-1})=-\lambda_{c^{-1}}(\delta(c))$. In particular,
  \[\delta(c^{-1}\hat c)=-\lambda_{c^{-1}}(\delta(c))+\lambda_{c^{-1}}(\delta(\hat c))=\lambda_{c^{-1}}(-\delta(c)+\delta(\hat c)).\]

  Suppose that $\rho(\delta(c))=\rho(\delta(\hat c))$. Then $-\delta(c)+\delta(\hat c)\in \ker\rho$ and, since $\ker \rho$ is a strong left ideal,  $\delta(c^{-1}\hat c)=\lambda_{c^{-1}}(-\delta(c)+\delta(\hat c))\in \ker \rho\subseteq \delta(\ker \beta)$. Hence, there exists $d\in \ker \beta$ such that $\delta(c^{-1}\hat c)=\delta(d)$. Since $\delta$ is bijective, $c^{-1}\hat c=d\in\ker \beta$. It follows that $\beta_c=\beta_{\hat c}$.

  \emph{\ref{en-rep-id-2} implies~\ref{en-rep-id-3}}. Condition~\ref{en-rep-id-2} implies that $\boxdot$ defined as in Equation~\eqref{eq-prod-br-rep} is a binary operation and $(\rho(K),\cdot,\boxdot)$ is a skew left brace with the desired lambda map. Furthermore, $\rho$ becomes a brace epimorphism. Therefore, $\ker\rho$ is an ideal of~$B$.
  
  \emph{\ref{en-rep-id-3} implies~\ref{en-rep-id-1}}. Suppose that $\ker\rho$ is an ideal. Let $\delta(c)\in\ker\rho$ and let $k\in K$. Since $\ker\rho$ is an ideal, by \cite[Lemma~3.7]{BallesterEsteban22}, $\lambda_c(k)-k\in\ker\rho$, hence $\rho(k)=\rho(\lambda_c(k))=\beta_c(\rho(k))$ for all $k\in K$. Therefore, $c\in\ker\beta$.\qedhere
\end{proof}

\begin{defi}\label{defi-br-rep}
Let $(B,+,\cdot)$ be a skew left brace with additive group $K$, multiplicative group $C$ and lambda map $\lambda$. A \emph{brace representation} of $B$ is a group representation $\rho\colon K\longrightarrow\GL(V)$ such that $\ker\rho$ is an ideal of $B$.
\end{defi}
\begin{remark}
If $\rho\colon K\longrightarrow\GL(V)$ is a brace representation, Proposition~\ref{prop-br-rep-equi} shows that $\ker\rho\subseteq\delta(\ker\beta)$. Consequently, $\rho(K)$ has a structure of skew left brace and $\rho$ is a brace homomorphism.
\end{remark}

\begin{ex}
Let $(G,\cdot,\cdot)$ be a trivial skew left brace, and consider a group representation $\rho\colon G\longrightarrow\GL(V)$. Since every normal subgroup of $G$ is an ideal, $\rho$ is a brace representation.  Hence, the brace representations of a trivial skew left brace coincide with the group representations.
\end{ex}

However, not every brace semi-representation is a brace representation.
\begin{ex}\label{ex-br-semi-rep-not-br-rep}
  Let $B$ be a skew left brace with additive group $K=\langle x\mid 6x=0\rangle$ and multiplication given by $(ax)(bx)=(a+(-1)^ab)x$ for $a$, $b\in\mathbb{Z}$. Consider the group representation $\rho\colon K\longrightarrow \C$ given by $\rho(kx)=\zeta^k$, where $\zeta=\mathrm{e}^{2\pi\mathrm{i}/3}$. Since $K$ is a cyclic group, every subgroup of $K$ is a strong left ideal of the skew left brace. Hence, $\rho$ is a brace semi-representation. The multiplicative group of~$B$ is isomorphic to the symmetric group of degree~$3$, which does not have normal subgroups of order~$2$. Therefore, $\ker \rho$ is not an ideal of~$B$, so $\rho$ is not a brace representation.

  We note that $\rho(3x)=\rho(0)=1$, but $\beta_{3x}(\rho(x))=\rho(\lambda_{3x}(x))=\rho(-x)=\zeta^2$, while $\beta_0(\rho(x))=\rho(\lambda_0(x))=\rho(x)=\zeta$. Hence, $\beta_{3x}\ne \beta_0$.
\end{ex}

\section{Trifactorised group representations}\label{sec-3fact-rep}
In this section, we will consider trifactorised groups in the sense of~\cite{BallesterEstebanPerezAPerezC25-arXiv-categoriesskewleftbraces}, that is, $4$-tuples of the form $(G, K, H, E)$, where $G=KH=KE=HE$, $K$ is a normal subgroup of~$G$, and $K\cap E=H\cap E=1$. We will define a representation of a generalised trifactorised group as a group representation $\rho\colon G\longrightarrow \GL(V)$ such that the kernel and the image are trifactorised groups compatible with the trifactorisation of $G$. By \cite[Proposition 5.9]{BallesterEstebanPerezAPerezC25-arXiv-categoriesskewleftbraces}, we can define a representation of a trifactorised group as follows:
\begin{defi}\label{defi-3fact-rep}
Let $(G,K,H,E)$ be a trifactorised group. A \emph{trifactorised group representation} is a group representation $\rho\colon G\longrightarrow\GL(V)$ satisfying 
\begin{equation}\label{eq-3fact-rep}
\ker\rho=(\ker\rho\cap K)H\cap(\ker\rho\cap K)E.
\end{equation}
\end{defi}

Note that this is a group representation of $G$ such that $\ker\rho$ is a trifactorised subgroup of $(G,K,H,E)$ in the sense of \cite[Definition~5.8]{BallesterEstebanPerezAPerezC25-arXiv-categoriesskewleftbraces}.

\begin{prop}\label{prop-Im-3fact-rep}
If $(G, K, H, E)$ is a trifactorised group and $\rho\colon G\longrightarrow \GL(V)$ is a trifactorised group representation, then $(\rho(G),\rho(K),\rho(H),\rho(E))$ is a trifactorised group.
\end{prop}
\begin{proof}
  Since $G=KH=KE=HE$, $\rho(G)=\rho(K)\rho(H)=\rho(K)\rho(E)=\rho(H)\rho(E)$. Since $K\trianglelefteq G$, $\rho(K)\trianglelefteq \rho(G)$. Suppose $y\in\rho(K)\cap \rho(E)$, then $y=\rho(k)=\rho(e)$ and so $ke^{-1}\in \ker\rho\subseteq (\ker\rho\cap K)E$. As $K\cap E=1$, we have that $k\in \ker\rho\cap K$, that is, $\rho(k)=1$. Consequently $y=1$ and $\rho(K)\cap \rho(E)=1$. Suppose now that $z\in\rho(H)\cap \rho(E)$. Then $z=\rho(h)=\rho(e)$ and so $he^{-1}\in\ker\rho\subseteq (\ker\rho\cap H)E$. As $H\cap E=1$, we have that $h\in\ker\rho\cap H$, that is, $\rho(h)=1$ and so $z=1$ and $\rho(H)\cap \rho(E)=1$.  
\end{proof}

The condition in Equation~\eqref{eq-3fact-rep} is necessary, because there are group representations of~$G$ for which Proposition~\ref{prop-Im-3fact-rep} does not hold.
\begin{ex}\label{ex-grp-rep-not-3fact-rep}
Consider a direct product $G=\gen{a}\x\gen{b}$ of two cyclic groups of order~$4$. If we define $K=\gen{a}$, $E=\gen{b}$, and $H=\gen{ab}$, we have that $(G,K,H,E)$ is a trifactorised group. In fact, it is the large trifactorised group associated with the trivial skew left brace with additive group $C_4$ (the one analysed in \cite{BallesterEsteban22}, see also \cite[Section~2]{BallesterEstebanPerezAPerezC25-arXiv-categoriesskewleftbraces}). Consider the group representation $\rho\colon G\longrightarrow\GL(\C^4)$ given by
\begin{align*}
\rho(a)=\begin{bmatrix}
-1 & 0 & 0 & 0\\
0 & -1 & 0 & 0\\
0 & 0 & -1 & 0\\
0 & 0 & 0 & -1 
\end{bmatrix},\quad \rho(b)=\begin{bmatrix}
0 & 0 & 0 & 1\\
1 & 0 & 0 & 0\\
0 & 1 & 0 & 0\\
0 & 0 & 1 & 0  
\end{bmatrix}.
\end{align*} 
Then $\ker\rho=\gen{a^2}$. However $(\ker\rho\cap K)H\cap (\ker\rho\cap K)E=\langle a^2\rangle \langle ab\rangle\cap \langle a^2\rangle \langle b\rangle=\gen{a^2, b^2, ab}\cap \langle a^2, b\rangle=\langle a^2, b^2\rangle\neq\ker\rho$. Furthermore, $\rho(a^2b^2)=\rho(a)^2\rho(b^2)=\rho(b^2)\ne 1$. Therefore, $\rho(H)\cap \rho(E)\neq 1$, which means that $\rho$ is not compatible with the trifactorisation.
\end{ex}

\section{Induced trifactorised group representations for finite braces}\label{sec-ind-3fact-rep}

In this section, we will consider only finite skew left braces and finite groups. 

 We have constructed trifactorised groups associated with skew left braces in \cite{BallesterEstebanPerezAPerezC25-arXiv-categoriesskewleftbraces}. The aim of this section is, given a finite skew left brace $B$, a brace representation $\rho$ of~$B$, and a trifactorised group $(G, K, H, E)$ associated with~$B$, to obtain a trifactorised group representation $\bar{\bar\rho}$ of $(G, K, H, E)$ compatible with~$\rho$.

Let $(G,K,H,E)$ be a trifactorised group associated with $B$ with corresponding epimorphism $\eta\colon C\longrightarrow E$.

\begin{defi}\label{def-induced-rep}
Consider a brace representation $\rho\colon K\longrightarrow\GL(V)$ of a finite skew left brace $B$ over a field $\F$. We say that a group representation $\bar{\bar\rho}\colon G\longrightarrow\GL(W)$, where $W$ is an $\F$-vector space, is an \emph{induced trifactorised group representation} of $\rho$ to $(G,K,H,E)$ if the following two conditions hold:
\begin{enumerate}[(\text{IR}1)]
\item $\ker\bar{\bar\rho}=(\ker\bar{\bar\rho}\cap K)H\cap(\ker\bar{\bar\rho}\cap K)E$.\label{IR1}
\item There exists a $K$-invariant subspace $V_0$ of $W$ such that:\label{IR2}
\begin{enumerate}
\item There exists an $\F K$-isomorphism $\alpha\colon V\longrightarrow V_0$, that is, $\bar{\bar\rho}(k)\circ \alpha=\alpha\circ\rho(k)$ for all $k\in K$, and
\item $W=\displaystyle\bigoplus_{\bar e\in \bar E}\bar e V_0$, where $\bar E= E\ker\bar{\bar\rho}/\ker\bar{\bar\rho}$. 
\end{enumerate}
\end{enumerate}
\end{defi}

Condition~\ref{IR1} comes from Equation~\eqref{eq-3fact-rep}, but we can also consider a stronger condition.
\begin{prop}\label{prop-ker=rho-ker-rho}
  If $\rho\colon K\longrightarrow \GL(V)$ is a brace representation of the finite skew left brace~$B$ over a field~$\mathbb{F}$, $(G, K, H, E)$ is a trifactorised group associated with~$B$,  and $\bar{\bar\rho}\colon G\longrightarrow \GL(W)$ is a group representation satisfying Condition~\ref{IR2} of Definition~\ref{def-induced-rep}, then 
$\ker\bar{\bar\rho}\cap K=\ker\rho$.
\end{prop}
\begin{proof}
Let $k\in\ker\bar{\bar\rho}\cap K$ and $v\in V$, then $\alpha(v)=\bar{\bar\rho}(k)(\alpha(v))=\alpha(\rho(k)(v))$.
Since $\alpha$ is bijective, we have that $v=\rho(k)(v)$. Thus, $k\in\ker\rho$.

On the other hand, consider $k\in\ker\rho$, $e\in E$ and $v\in V$. Since $\ker\rho$ is an ideal, $\ker\rho$ is normal in $G$. Therefore,  there exists $k'\in \ker\rho$ such that $k=ek'e^{-1}$. Then 
\begin{align*}
\bar{\bar\rho}(k)(e\bullet \alpha(v))&=\bar{\bar\rho}(ke)(\alpha(v))=\bar{\bar\rho}(ek')(\alpha(v))=\alpha(\rho(e)(\rho(k')(v)))\\&=\alpha(\rho(e)(v))=\bar{\bar\rho}(e)(\alpha(v))=e\bullet\alpha(v).
\end{align*}
Since $W=\bigoplus_{\bar e\in\bar E}\bar eV_0$, it follows that $k\in\ker\bar{\bar\rho}$.\qedhere
\end{proof}

\begin{remark}
By Proposition~\ref{prop-ker=rho-ker-rho} we can replace Condition~\ref{IR1} of Definition~\ref{def-induced-rep} by the following condition:
\begin{enumerate}[(\text{IR}1')]
\item $\ker\bar{\bar\rho}=(\ker\rho)H\cap(\ker\rho)E$. \label{IR1'}
\end{enumerate}
\end{remark}

Let us suppose that $\rho$ is a brace representation of~$B$. Take $T=(\ker\rho)H\cap(\ker\rho)E$. Since $K\cap T=\ker\rho$, we have that the decomposition of an element $kt$ of $KT$ as a product of $k\in K$ and $t\in T$ is unique and so $\bar\rho\colon KT\longrightarrow\GL(V)$ given by $\bar\rho(kt)=\rho(k)$ for all $k\in K$ and $t\in T$ is a map. Let $k$, $k_1\in K$ and $t$, $t_1\in T$. Since $\ker\rho$ is an ideal, $T$ is a normal subgroup of $G$ by \cite[Proposition~5.2~(4)]{BallesterEstebanPerezAPerezC25-arXiv-categoriesskewleftbraces}. It follows that $k_1^{-1}tk_1t^{-1}\in [K, T]\subseteq K\cap T=\ker\rho$ by \cite[Kapitel~III, Hilfssatz~1.6]{Huppert67}, and so $\rho(k_1)=\rho(tk_1t^{-1})$. Hence $\bar\rho(ktk_1t_1)=\bar\rho(k(tk_1t^{-1})tt_1)=\rho(k(tk_1t^{-1}))=\rho(k)\rho(tk_1t^{-1})=\rho(k)\rho(k_1)=\bar\rho(kt)\bar\rho(k_1t_1)$ and so $\bar\rho$ is a group homomorphism, that is, $V$ is an $\F (KT)$-module via the action $kt\RHD v=\bar\rho(kt)(v)=\rho(k)(v)$. Furthermore, $1=\bar \rho(kt)=\rho(k)$ if, and only if, $k\in \ker\rho$, that is, $\ker\bar\rho=(\ker \rho)T=T$.
Moreover, $\F G$ is an $(\F G,\F(KT))$-bimodule via the multiplication action. Therefore, we can consider the $\F G$-module $W=\F G\otimes_{\F (KT)}V$ given by $g\bullet (\omega\otimes v)=g\omega\otimes v.$ This gives us a group representation of $G$,
\begin{equation}
  \bar{\bar\rho}\colon G\longrightarrow\GL(W),\quad \bar{\bar\rho}(g)(\omega\otimes v)=g\omega\otimes v.\label{eq-bar-bar-rho}
\end{equation}

\begin{prop}\label{prop-ind-3fact-rep-exist}
If $\rho$ is a brace representation of a finite skew left brace~$B$ and $(G, K, H, E)$ is a trifactorised group associated with $B$, then $\bar{\bar\rho}$ defined as in Equation~\eqref{eq-bar-bar-rho} is a trifactorised group representation of $(G,K,H,E)$. Moreover, in this case, $\bar{\bar\rho}$ is an induced trifactorised group representation of $\rho$ to $(G,K,H,E)$.
\end{prop}
\begin{proof}
  We know that $\ker\rho=T\cap K$ and $\ker\bar\rho=T$. 
  By \cite[Chapter~B, Proposition~6.4]{DoerkHawkes92}, $\ker\bar{\bar\rho}=\Core_G(\ker\bar\rho)=\Core_G(T)=T$. Since $T=(\ker \rho)H\cap (\ker \rho)E=(T\cap K)H\cap (T\cap K)E$, we have that $T=\ker\bar{\bar \rho}$ is a trifactorised group representation of~$B$. In particular, $\bar{\bar\rho}$ satisfies Condition~\ref{IR1}.

Let us prove that it satisfies Condition~\ref{IR2}, too. Take $V_0=1\otimes V$, which is an $\F K$-module, and $\alpha\colon V\longrightarrow V_0$ given by $\alpha(v)=1\otimes v$, which is an $\F K$-isomorphism.

Let $\bar E=ET/T=\{\bar e_1,\dots, \bar e_n\}$ where $\bar e_j=e_jT$ for some $e_j\in E$, $1\le j\le n$. Since $G=KE$, we have that $G/T=(KT/T)(ET/T)$ and
\begin{align*}
    ET\cap KT&=(E\cap K((\ker\rho)H\cap (\ker\rho)E))T\\&=(E\cap K(\ker\rho)((\ker\rho)H\cap E))T=(E\cap K((\ker\rho)H\cap E))T\\
    &=(E\cap K)((\ker\rho)H\cap E)T=T
\end{align*}
by a repeated application of the Dedekind identity and so $ET/T$ is a transversal of $KT/T$ in $G/T$. It follows that $\{e_1,\dots, e_n\}$ is a transversal of $KT$ in $G$. Then by \cite[Chapter~B, Equation~(6.$\alpha$)]{DoerkHawkes92},
\[W=\bigoplus_{i=1}^ne_i\otimes V=\bigoplus_{i=1}^ne_iV_0=\bigoplus_{i=1}^n\bar e_iV_0.\]

 Hence, $\bar{\bar\rho}$ satisfies Condition~\ref{IR2}.
\end{proof}

\begin{defi}
  Given two trifactorised group representations $\bar{\bar\rho}\colon G\longrightarrow \GL(W)$ and $\bar{\bar\rho}'\colon G\longrightarrow \GL(W)$, we say that they are \emph{equivalent} if there exists a linear map $\varphi\colon W\longrightarrow W$ such that $\varphi\circ\bar{\bar\rho}(g)=\bar{\bar\rho}'(g)\circ\varphi$, in other words, when there exists an $\F G$-isomorphism $\varphi\colon V\longrightarrow W$.
\end{defi}

\begin{prop}\label{prop-ind-3fact-rep-equi}
Let $\rho$ be a brace representation of a finite skew left brace~$B$ and let $(G, K, H, E)$ be a trifactorised group associated with~$B$. Then all induced trifactorised group representations of $\rho$ to $(G,K,H,E)$ are equivalent.
\end{prop}
\begin{proof}
Consider $\bar{\bar\rho}\colon G\longrightarrow\GL(W)$ constructed in Proposition~\ref{prop-ind-3fact-rep-exist} and consider another induced representation $\bar{\bar\rho}'\colon G\longrightarrow\GL(W')$ where $V_0'$ is the $K$-invariant subspace of $W'$ and $\alpha'\colon V\longrightarrow V_{0}'$ is the corresponding $\F K$-isomorphism. We want to show that $\bar{\bar\rho}$ and $\bar{\bar\rho}'$ are equivalent representations, which is equivalent to prove that $W$ and $W'$ are isomorphic $\F G$-modules. 

Since $\ker\bar{\bar\rho}'=(\ker\rho) H\cap(\ker\rho) E=\ker\bar{\bar\rho}$, we have that
\[\bar E=E\ker\bar{\bar\rho}/\ker\bar{\bar\rho}=E\ker\bar{\bar\rho}'/\ker\bar{\bar\rho}'.\]
Therefore, we can define $\varphi\colon W'\longrightarrow W$ given by $\varphi(\bar e\bullet \alpha'(v))=\bar e\otimes v$ and extended by linearity. We know that $W'=\bigoplus_{\bar e\in\bar E}\bar e V_0'$, $W=\bigoplus_{\bar e\in\bar E}\bar eV_0$ and $\alpha'$ is bijective, therefore $\varphi$ is a bijective map. Furthermore, given $g\in G$, $e\in E$, and $v\in V$ there exists $k'\in K$ and $e'\in E$ such that $ge=e'k'$. Then
\begin{align*}
\varphi(g\bullet (\bar e\bullet \alpha'(v)))&=\varphi(\bar{\bar\rho}'(ge)(\alpha'(v)))=\varphi(\bar{\bar\rho}'(e'k')(\alpha'(v)))\\&=\varphi(\bar{e}'\bullet(\bar{\bar\rho}'(k')(\alpha'(v))))=\varphi(\bar{e}'\bullet\alpha'(\rho(k')(v)))\\&=\bar{e}'\otimes \rho(k')(v)=\bar{e}'\otimes (\bar{k}'\bullet v)=\overline{e'k'}\otimes v=\overline{ge}\otimes v\\&=g\bullet (\bar{e}\otimes v)=g\bullet\varphi(\bar e\bullet\alpha'(v))
\end{align*}
In conclusion, $\varphi$ is an $\F G$-isomorphism and $\bar{\bar\rho}$ and $\bar{\bar\rho}'$ are equivalent representations. Therefore, every induced trifactorised group representation of $\rho$ to $(G,K,H,E)$ is equivalent to $W$.
\end{proof}

Propositions~\ref{prop-ind-3fact-rep-exist} and~\ref{prop-ind-3fact-rep-equi} allow us to define the following:
\begin{defi}\label{defi-ind-3fact-rep}
The \emph{induced trifactorised group representation} of the brace representation $\rho\colon B\longrightarrow \GL(V)$ of the finite skew left brace $B$ to the trifactorised group $(G,K,H,E)$ associated with~$B$ is
\[\bar{\bar\rho}\colon G\longrightarrow \GL(W),\]
where
\[W=\F G\otimes_{\F(KT)} V,\]
given by
$\bar{\bar\rho}(g)(\omega\otimes v)=g\omega\otimes v$ for $\omega\in \F G$ and $v\in V$,
where $T=(\ker\rho)H\cap (\ker\rho)E$, $V_0=1\otimes V$, and $\alpha\colon V\longrightarrow V_0$ is given by $\alpha(v)=1\otimes v$.
\end{defi}

\begin{prop}\label{prop-ind-3fact-rep-beta}
  Suppose that $\rho\colon K\longrightarrow \GL(V)$ is a brace representation of a skew left brace~$B$ with associated beta map $\beta$, $(G, K, H, E)$ is a trifactorised group associated with~$B$ with associated group epimorphism $\eta\colon C\longrightarrow E$, and $\bar{\bar\rho}\colon G\longrightarrow \GL(W)$ is the corresponding induced trifactorised group representation with associated $\F K$-isomorphism $\alpha$. Then, for every $c\in C$ and $k\in K$,
\[\alpha\circ\beta_c({\rho(k)})=({\bar{\bar\rho}(\eta(c))\circ\bar{\bar\rho} (k)\circ \bar{\bar\rho}(\eta(c))^{-1}})\circ\alpha.\]
\end{prop}
\begin{proof}
Consider $k\in K$, $c\in C$ and $v\in V$ then
\begin{align*}
({\bar{\bar\rho}(\eta(c))\circ\bar{\bar\rho} (k)\circ \bar{\bar\rho}(\eta(c))^{-1}})(\alpha(v))&=\bar{\bar\rho} ({\eta(c)k\eta(c)^{-1}})(\alpha(v))\\&=\alpha({\rho({\eta(c)k\eta(c)^{-1}})(v)})\\&=\alpha({\rho({\lambda_c(k)})(v)})\\&=\alpha({\beta_c({\rho(k)})(v)}).\qedhere
\end{align*}
\end{proof}
\begin{remark}
Proposition~\ref{prop-ind-3fact-rep-beta} shows that $\beta_c$ can be viewed as the conjugation in $\bar{\bar\rho}(K)$ by the element $\bar{\bar\rho}(\eta(c))$. Therefore, $\beta$ is encoded as the conjugation by the elements of $\bar{\bar\rho}(E)$ in $\bar{\bar\rho}(K)$, in a similar way as $\lambda$ is encoded in the trifactorised group as the conjugation by the elements in $E$ in $K$. This is the expected behaviour, as $\beta$ encodes $\lambda$ in the brace representation.
\end{remark}

\begin{prop}\label{prop-3fact-to-rep-Br-rep}
Consider a group representation $\bar{\bar\rho}\colon G\longrightarrow\GL(V)$, where $(G, K, H, E)$ is a trifactorised group associated with a finite skew left brace~$B$. Then:
\begin{enumerate}
\item $\rho=\bar{\bar\rho}|_K$ is a brace semi-representation of $B$ with beta map $\beta$ satisfying that $\beta_c$ corresponds to the conjugation by $\bar{\bar\rho}({\eta(c)})$ in $\rho(K)$ for all $c\in C$.
\item  $\bar{\bar\rho}$ is a trifactorised group representation of $(G,K,H,E)$ if and only if, $\rho$ is a brace representation of $B$.
\end{enumerate}
\end{prop}
\begin{proof}
  \begin{enumerate}
  \item 
We have that $\rho$ is a group representation of the additive group $K$ of $B$, with kernel $\ker\bar{\bar\rho}\cap K$. Since $\ker\bar{\bar\rho}\cap K$ is normal in $G$, it follows that it is lambda-invariant, therefore, it is a strong left ideal, and $\rho$ is a semi-representation of $B$. If $\beta\colon C\longrightarrow\Aut(\rho(K))$ is the beta map of $\rho$, it follows that $\beta_c(\rho(k))=\rho(\lambda_c(k))=\rho({\eta(c)k\eta(c)^{-1}})=\bar{\bar\rho}(\eta(c))\rho(k)\bar{\bar\rho}(\eta(c))^{-1}$.

\item
Moreover, $\bar{\bar\rho}$ is a trifactorised group representation if and only if, $(\ker\rho) H\cap(\ker\rho) E=\ker\bar{\bar\rho}$ is a normal subgroup of $G$ that is, $\ker\rho$ is an ideal of $B$, or, equivalently, $\rho$ is a brace representation. \qedhere
\end{enumerate}
\end{proof}

\section{ZLV-representations and modules}\label{sec-ZLV-rep}

In \cite[Remark~3.3]{Zhu22} for skew left braces with abelian additive group and in~\cite[Definition~4.1]{LetourmyVendramin24} for general skew left braces, a representation of a skew left brace~$B$, called here a \ZLV-representation to distinguish it from our representations, was defined in the following way:
\begin{defi}\label{defi-ZLV}
  A \emph{\ZLV-representation} of a skew left brace~$B$ with additive group $K$ and multiplicative group~$C$ is a pair $(\rho, \beta)$ of group representations $\rho\colon K\longrightarrow\GL(V)$ and $\beta\colon C\longrightarrow \GL(V)$, where $V$ is a vector space, such that
  \[\rho(\delta(xy))=\beta(x)\rho(\delta(y))\beta(x)^{-1}\rho(\delta(x))\]
  for all $x$, $y\in C$. The vector space $V$ is called the \emph{representation space} of $(\rho, \beta)$.
\end{defi}

We face some difficulties related to the lack of conditions ensuring that the brace homomorphism associated with this representation has an ideal as its kernel.
\begin{ex}\label{ex-ZLV-rep-not-br-rep}
 Let $B$ be the left brace where $B=\Z/6\Z=\{0,\ldots,5\}$ with the usual sum and the multiplication given by $ab=a+(-1)^ab$ for $a$, $b\in B$.  We note that the multiplicative group $(B, {\cdot})$ is isomorphic to the symmetric group $\Sigma_3$ of degree~$3$. Let $\zeta$ be a complex primitive cube root of unity and let $\rho$, $\beta\colon B\longrightarrow \GL_2(\C)$ be given by
\begin{align*}
  \rho(x)= \begin{bmatrix}
    \zeta^x&0\\
    0&\zeta^{2x}
  \end{bmatrix},\quad 
  \beta(x)= \begin{bmatrix}
    0&1\\
    1&0
  \end{bmatrix}^x,\quad 0\leq x\leq 5.
  \end{align*}
  It is easy to check that $(\rho,\beta)$ is a \ZLV-representation of $B$. However, $\ker\rho=\{0,3\}$ is not a normal subgroup of $(B,\cdot)$. Moreover, $\rho(3)=1$, while $\beta(3)\ne 1$.
\end{ex}
Our main concern here is that in Definition~\ref{defi-ZLV} there is no relation between the kernels of both homomorphisms as in Condition~\ref{en-rep-id-1} of Proposition~\ref{prop-br-rep-equi} and there is no proviso making the kernel of $\rho$ an ideal of $B$. We see now that \ZLV-representations can be regarded as brace semi-representations.

\begin{prop}\label{prop-ZLV-rep-ind-br-semi-rep}
Consider a \ZLV-representation $(\rho, \beta)$ of a skew left brace $B$ over a vector space $V$. Then $\rho$ is a brace semi-representation of $B$ and its beta map is given by the conjugations by $\rho(\delta(c))^{-1}\beta(c)$ in $\rho(K)$ for all $c\in C$.
\end{prop}
\begin{proof}
Let $c$, $c'\in C$. Then 
\begin{align*}
\rho({\lambda_c(\delta(c'))})&=\rho(\delta(c)^{-1}\delta(cc'))=\rho(\delta(c)^{-1}\delta(cc')\delta(c)^{-1}\delta(c))\\&=\rho(\delta(c))^{-1}\rho(\delta(cc')\delta(c)^{-1})\rho(\delta(c))\\&=\rho(\delta(c))^{-1}\beta(c)\rho(\delta(c'))\beta(c)^{-1}\rho(\delta(c)).
\end{align*}
If $k\in\ker\rho$, then $\rho(\lambda_c(k))=1$ for all $c\in C$. Therefore, $\ker\rho$ is lambda-invariant, which means that it is a strong left ideal. Hence, $\rho$ is a brace semi-representation of $B$. Let $\bar\beta\colon C\longrightarrow\Aut(\rho(K))$ be its beta map. Then
$\bar\beta_c(\rho(k))=\rho(\lambda_c(k))$ that, as stated before, corresponds to the conjugation of $\rho(k)$ by $\rho(\delta(c))^{-1}\beta(c)$.
\end{proof}

There is no converse of Proposition~\ref{prop-ZLV-rep-ind-br-semi-rep}, in other words, not every brace semi-representation comes from a \ZLV-representation, as we will see in Example~\ref{ex-br-semi-rep-not-ZLV-rep} below. Therefore, we cannot make an ``inverse'' construction. 
Another problem is that there exist brace semi-representations that come from multiple \ZLV-representations, see Example~\ref{ex-br-semi-rep-var-ZLV-rep}.

\begin{ex}\label{ex-br-semi-rep-not-ZLV-rep}
Consider the left brace of Example~\ref{ex-ZLV-rep-not-br-rep} and consider the brace representation $\rho\colon K\longrightarrow\GL(\F_7^2)$ given by
\[\rho(x)=\begin{bmatrix}
0&1\\1&2
\end{bmatrix}^x\]
for $0\le x\le 5$.
If $\bar\beta\colon C\longrightarrow\Aut(\rho(K))$ is the beta map, we have that $\bar\beta_{1}(\rho(x))=\rho(x)^{-1}$ for $0\le x\le 5$.

Let suppose that $\rho$ comes from a \ZLV-representation $(\rho,\beta)$. Then, by Proposition~\ref{prop-ZLV-rep-ind-br-semi-rep}, $\bar\beta_1$ must be a conjugation by some element of $\F_7^2$ restricted to $\rho(K)$. Therefore,
\[\rho(1)=\begin{bmatrix}
0&1\\1&2\\
\end{bmatrix}\text{ and }\bar\beta_1({\rho(1)})=\rho(1)^{-1}=\begin{bmatrix}
5&1\\1&0\\
\end{bmatrix}\]
are conjugate matrices, that is, similar matrices, which is false, as they have different characteristic polynomials. Therefore, $\rho$ is a brace representation that does not come from a \ZLV-representation. 
\end{ex}

\begin{ex}\label{ex-br-semi-rep-var-ZLV-rep}
Consider the trivial skew left brace $B$ associated with $C_2=\gen{x}$. Let $\rho$, $\beta$, $\beta'\colon C_2\longrightarrow\GL(\F_2^2)$ be given by
\[\rho(x)=\begin{bmatrix}
1&0\\1&1
\end{bmatrix},\quad \beta(x)=\begin{bmatrix}
1&0\\0&1
\end{bmatrix},\quad \beta'(x)=\begin{bmatrix}
1&0\\1&1
\end{bmatrix}.\]
It is easy to check that $(\rho,\beta)$ and $(\rho,\beta')$ are \ZLV-representations of~$B$.
\end{ex}

Let us see how \ZLV-representations are related to the trifactorised groups.

\begin{prop}\label{prop-3fact-rep-ind-ZLV}
  Let $B$ be a finite skew left brace with additive group $K$ and multiplicative group~$C$. Let $(G, K, H, E)$ be a trifactorised group associated with~$B$. Let $\delta\colon C\longrightarrow K$ be the corresponding bijective derivation and let $\eta\colon C\longrightarrow E$ be the associated group epimorphism. Let $\sigma\colon C\longrightarrow H$ given by $\sigma(c)=\delta(c)\eta(c)$. Let $\bar{\bar\rho}\colon G\longrightarrow\GL(V)$ be a group representation, $\rho=\bar{\bar\rho}|_K$, and $\beta=\bar{\bar\rho}|_H\circ\sigma$. Then $(\rho,\beta)$ is a \ZLV-representation. 
\end{prop}
\begin{proof}
Given $x$, $y\in C$, we have that
\begin{align*}
\rho(\delta(xy))&=\rho({\delta(x)\eta(x)\delta(y)\eta(x)^{-1}})=\bar{\bar\rho}({\delta(x)\eta(x)})\rho(\delta(y))\bar{\bar\rho}(\eta(x))^{-1}\\&=\beta(x)\rho(\delta(y))\bar{\bar\rho}(\delta(x)\eta(x))^{-1}\rho(\delta(x))=\beta(x)\rho(\delta(y))\beta(x)^{-1}\rho(\delta(x)).
\end{align*}
Hence, $(\rho,\beta)$ is a \ZLV-representation.
\end{proof}

A somehow converse of Proposition~\ref{prop-3fact-rep-ind-ZLV} can be attained. If $(\rho,\beta)$ is a \ZLV-representation of $B$, then, by Proposition~\ref{prop-ZLV-rep-ind-br-semi-rep} we have that $\rho$ is a brace semi-representation of $B$. Then we can induce it to a trifactorised group $(G,K,H,E)$ associated with~$B$ as in Proposition~\ref{prop-ind-3fact-rep-exist}. Therefore, given a \ZLV-representation, we can induce a trifactorised representation. But in the case of the large trifactorised group (as in~\cite{BallesterEsteban22}, see also \cite[Section~2]{BallesterEstebanPerezAPerezC25-arXiv-categoriesskewleftbraces}) we obtain a more interesting result.

\begin{prop}\label{prop-ZLV-ind-rep-large}
Let $\fL(B,+,\cdot)=([K]C,K,D,C)$ be the large trifactorised group associated with $B$. Consider a \ZLV-representation $(\rho,\beta)$ of~$B$ over a vector space $V$. Then $\bar{\bar\rho}\colon [K]C\longrightarrow \GL(V)$ given by $\bar{\bar\rho}(kc)=\rho(k\delta(c)^{-1})\beta(c)$ for $k\in K$ and $c\in C$ is a group representation of $[K]C$. 
\end{prop}
\begin{proof}
  First, note that, by Proposition~\ref{prop-ZLV-rep-ind-br-semi-rep},
  \[\rho(\lambda_c(k))=\rho(\delta(c))^{-1}\beta(c)\rho(k)\beta(c)^{-1}\rho(\delta(c))\]
  for $k\in K$ and $c\in C$. Therefore, given $k_1$, $k_2\in K$ and $c_1$, $c_2\in C$,
\begin{align*}
  \bar{\bar\rho}({({k_1c_1})({k_2c_2})})&=\bar{\bar\rho}({k_1\lambda_{c_1}(k_2)c_1c_2})=\rho({k_1\lambda_{c_1}(k_2)\delta(c_1c_2)^{-1}})\beta(c_1c_2)\\&=\rho({k_1\lambda_{c_1}(k_2)})\rho({\delta(c_1c_2)})^{-1}\beta(c_1)\beta(c_2)\\&=\rho({k_1\lambda_{c_1}(k_2)})({\beta(c_1)\rho(\delta(c_2))\beta(c_1)^{-1}\rho(\delta(c_1))})^{-1}\beta(c_1)\beta(c_2)\\&=\rho({k_1\lambda_{c_1}(k_2)})\rho(\delta(c_1))^{-1}\beta(c_1)\rho(\delta(c_2))^{-1}\beta(c_2)\\&=\rho({k_1}) \rho(\delta(c_1))^{-1}\beta(c_1)\rho(k_2))\rho(\delta(c_2))^{-1}\beta(c_2)\\&=\rho(k_1\delta(c_1)^{-1})\beta(c_1)\rho(k_2\delta(c_2)^{-1})\beta(c_2)=\bar{\bar\rho}(k_1c_1)\bar{\bar\rho}(k_2c_2).
\end{align*}
Hence, $\bar{\bar\rho}$ is a group representation of $[K]C$.
\end{proof}
\begin{remark}
For the large trifactorised groups, the constructions of Propositions~\ref{prop-3fact-rep-ind-ZLV} and~\ref{prop-ZLV-ind-rep-large} are inverses of each other. Therefore, \ZLV-representations of a skew left brace are equivalent to group representations of the large trifactorised group associated with the skew left brace.
\end{remark}

In \cite[Definition~4.12]{LetourmyVendramin24}, the authors define the complex \emph{twisted skew brace algebra} of a skew left brace $B$ with respect to a $2$-cocycle $(\alpha, \mu)$ of a skew left brace. They proved in \cite[Proposition~4.19]{LetourmyVendramin24} that complex projective representations (defined with image in $\PGL_n(\C)$ instead of $\GL_n(\C)$) associated with the parameters $(\alpha, \mu)$ are in a bijective correspondence with some modules for this algebra. Our linear representations correspond to the case of the $2$-cocycle $(\alpha, \mu)$ with $\alpha$, $\mu$ the constant maps with image~$1$. Since we are interested in ordinary linear representations, with images in $\GL_n(\C)$, we will consider only this case.

\begin{defi}[see {\cite[Definition~4.12]{LetourmyVendramin24}}]
  Let $B$ be a skew left brace and consider the vector space $\C B$ with basis $\bar B=\{e_b\mid b\in B\}$. For $a$, $b\in B$, let us define
\[e_a\oplus e_b=e_{a+b},\qquad e_a\odot e_b=e_{ab}\]
and extend these operations bilinearly on $\C B$. Then $\oplus$ and $\odot$ give two algebra structures to $\C B$, where $\ominus e_b=e_{-b}$ is the symmetric element of $e_b$ for $\oplus$ and $(e_b)^{\odot}=e_{b^{-1}}$ is the symmetric element of $e_b$ for $\odot$. This structure will be called the \emph{twisted skew brace algebra} of~$B$, and will be denoted as $\tsba{\C B}$. Furthermore, $\xi\colon B\longrightarrow \bar B$ given by $\xi(b)=e_b$ for $b\in B$ can be regarded as a group homomorphism from $(B, {+})$ to $(\bar B, {\oplus}|_{\bar B\times \bar B})$ and from $(B, {\cdot})$ to $(\bar B, {\odot}|_{\bar B\times \bar B})$, where ${\oplus}|_{\bar B\times \bar B}$ and ${\odot}|_{\bar B\times \bar B}$ represent the restrictions of $\oplus$ and $\odot$, respectively, to $\bar B\times \bar B$.
\end{defi}
\begin{defi}[see {\cite[Definition~4.15]{LetourmyVendramin24}}]
A \emph{module} over the twisted skew brace algebra $\tsba{\C B}$ is a $\C$-vector space $V$ together with two algebra homomorphisms
\[\hat\rho\colon (\tsba{\C B}, {+}, {\oplus})\longrightarrow \End(V),\qquad
  \hat\beta\colon (\tsba{\C B}, {\cdot}, {\odot})\longrightarrow \End(V)\]
satisfying the relation
\begin{equation}
  \hat\rho(e_a\odot e_b)=\hat\beta(e_a)\hat\rho(e_b)\hat\beta(e_a)^{-1}\hat\rho(e_a)\label{eq-module}
\end{equation}
for all $a$, $b\in B$.
\end{defi}
For every $a\in B$, since $e_a$ has symmetric elements for $\oplus$ and $\odot$, we have that $\rho(e_a)$ and $\beta(e_a)$ are invertible matrices. 
The composition of the restrictions $\hat\rho|_{\bar B}$ of $\hat \rho$ to $\bar B$ and $\hat\beta|_{\bar B}$ of $\hat \beta$ to $\bar B$ satisfy that $\rho=\hat\rho|_{\bar B}\circ \xi\colon (B, {+})\longrightarrow \GL(V)$ and $\beta=\hat \beta|_{\bar B}\circ \xi\colon (B, {\cdot})\longrightarrow \GL(V)$ are group homomorphisms. In other words, $\rho$ and $\beta$ are group representations of $(B, {+})$ and $(B, {\cdot})$, respectively. Moreover, Equation~\eqref{eq-module} can be rewritten as
\[\rho(ab)=\beta(a)\rho(b)\beta(a)^{-1}\rho(a)\]
for every $a$, $b\in B$. By fixing a basis on $V$, we can identify $\GL(V)$ with $\GL_n(\C)$ for $n=\dim_{\C}(V)$. It follows that $(\rho, \beta)$ is a \ZLV-representation of $B$ with representation space $V$. By Proposition~\ref{prop-ZLV-rep-ind-br-semi-rep}, $\rho$ is a brace semi-representation of~$B$ and, by Proposition~\ref{prop-ZLV-ind-rep-large}, $\bar{\bar \rho}\colon [K]C\longrightarrow \GL(V)$ given by $\bar{\bar \rho}(kc)=\rho(k\delta(c)^{-1})\beta(c)$ for $k\in K=(B,{+})$ and $c\in C=(B, {\cdot})$ is a group representation of $G=[K]C$. The arguments of \cite[Chapter~B, Section~3]{DoerkHawkes92} show:
\begin{prop}\label{prop-ZLV-large}
  Every module over the twisted skew brace algebra $\tsba{\C B}$ of a finite skew left brace $B$ can be regarded as a module over the semidirect product $G=[K]C$ of the additive group $K=(B, {+})$ and the multiplicative group $C=(B, {\cdot})$.
\end{prop}

\section*{Acknowledgements}
This work has been supported by the grant PID2024-159495NB-I00, funded by MICIU/AEI/10.13039/501100011033 and by ERDF/EU, and by the grant CIAICO/2023/007 from the
Conselleria d’Educació, Cultura, Universitats i Ocupació, of the Generalitat Valenciana.

The authors declare that there is no relevant competing interest.
\bibliographystyle{plain}
\bibliography{bibgroup}

@preamble{"
\providecommand{\iflanguage}[3]{#3}\relax
\providecommand{\mathfrak}{\mathcal}
\providecommand{\acceptedin}{\iflanguage{spanish}{Aceptado en}{Accepted in}}
\providecommand{\accepted}{\iflanguage{spanish}{Aceptado}{Accepted}}
\providecommand{\talkat}{\iflanguage{spanish}{Charla en}{Talk at}}
\providecommand{\plenarytalkat}{\iflanguage{spanish}{Charla plenaria en}{Plenary talk at}}
\providecommand{\JIF}{JIF}
\providecommand{\PhDcourseat}{\iflanguage{spanish}{Curso de doctorado en}{PhD course at}}
\providecommand{\transrus}{\iflanguage{spanish}{Traducci{\'o}n del art{\'\i}culo original en ruso en}{Translation of the original paper in Russian from}}
\providecommand{\supervisedby}{\iflanguage{spanish}{Direcci{\'o}n:}{Supervised by}}
\providecommand{\visited}{\iflanguage{spanish}{visitada}{visited}}
\providecommand{\inpress}{\iflanguage{spanish}{en prensa}{\iflanguage{catalan}{en premsa}{in press}}}
\providecommand{\encurs}{\iflanguage{spanish}{en curso}{\iflanguage{catalan}{en curs}{in course}}}
\providecommand{\aand}{\iflanguage{spanish}{y}{and}}
\providecommand{\invitedtalkat}{\iflanguage{spanish}{Conferencia invitada en}{Invited talk at}}
\providecommand{\invitedtalkatandorganiser}{\iflanguage{spanish}{Organizador y conferencia invitada en}{Organiser and invited talk at}}
\providecommand{\posterat}{\iflanguage{spanish}{Póster en}{Poster at}}
\providecommand{\invitedtalk}{\iflanguage{spanish}{Conferencia invitada}{Invited talk}}
\providecommand{\oral}{\iflanguage{spanish}{Comunicaci{\'o}n oral}{Oral communication}}
\providecommand{\Poland}{\iflanguage{spanish}{Polonia}{Poland}}
\providecommand{\Germany}{\iflanguage{spanish}{Alemania}{Germany}}
\providecommand{\Italy}{\iflanguage{spanish}{Italia}{Italy}}
\providecommand{\Ireland}{\iflanguage{spanish}{Irlanda}{Ireland}}
\providecommand{\Hungary}{\iflanguage{spanish}{Hungr{\'\i}a}{Hungary}}
\providecommand{\Portugal}{\iflanguage{spanish}{Portugal}{Portugal}}
\providecommand{\Cairo}{\iflanguage{spanish}{El Cairo, Egipto}{Cairo, Egypt}}
\providecommand{\Bath}{\iflanguage{spanish}{Bath, Inglaterra, Reino Unido}{Bath, England, United Kingdom}}
\providecommand{\presentada}[2]{\iflanguage{spanish}{presentada por #1}{presented by #2}}
\providecommand{\Ponenciaplenaria}{\iflanguage{spanish}{Ponencia plenaria}{Plenary talk}}
\providecommand{\Naples}{\iflanguage{spanish}{N{\'a}poles}{Naples}}
\def\cprime{$'$}
\providecommand{\germ}{\mathfrak}
\providecommand{\url}{\texttt}"}

@string{and = { and }}

@string{berlin = {Berlin, Heidelberg, New York}}

@string{grund = {Grund.\ Math.\ Wiss.}}

@string{jalg = {J.\ Algebra}}

@string{springer = {Springer Verlag}}

@Article{BallesterEsteban22,
  author =	 {A. Ballester-Bolinches and R. Esteban-Romero},
  title =	 {Triply factorised groups and the structure of skew
                  left braces},
  journal =	 {Commun. Math. Stat.},
  year =	 2022,
  volume =	 10,
  month =	 oct,
  pages =	 {353--370},
  doi =		 {10.1007/s40304-021-00239-6},
  impact =	 {SCI {\JIF} 0,868 (2021, MATHEMATICS 193/333,
                  Q3).{\newblock}Scopus CiteScore 2,7 (2021, Applied
                  Mathematics 82/250, Q2; Computational Mathematics
                  72/167, Q2; Statistics and Probability 82/250,
                  Q2).{\newblock}Scopus SJR 0,725 (2021, Applied
                  Mathematics 185/550, Q2; Computational Mathematics
                  53/154, Q2; Statistics and Probability 90/248, Q2).}
}

@unpublished{BallesterEstebanPerezAPerezC25-arXiv-categoriesskewleftbraces,
  title =	 {Categories of skew left braces and trifactorised
                  groups},
  author =	 {A. Ballester-Bolinches and R. Esteban-Romero and
                  P. Pérez-Altarriba and V. Pérez-Calabuig},
  year =	 2025,
  eprint =	 {2501.16089},
  archivePrefix ={arXiv},
  primaryClass = {math.GR},
  doi =		 {10.48550/arXiv.2501.16089},
  url =		 {https://arxiv.org/abs/2501.16089},
  note =	 {Preprint arXiv:2501.16089}
}

@Article{Baxter73,
  author =	 {R. Baxter},
  title =	 {Eight-vertex model in lattice statistics and
                  one-dimensional anisotropic {H}eisenberg
                  chain. {I}. {S}ome fundamental eigenvectors},
  journal =	 {Ann.\ Physics},
  year =	 1973,
  volume =	 76,
  number =	 1,
  month =	 mar,
  pages =	 {1--24},
  doi =		 {10.1016/0003-4916(73)90439-9}
}

@book{DoerkHawkes92,
  AUTHOR =	 {Doerk, K. and Hawkes, T.},
  TITLE =	 {Finite soluble groups},
  SERIES =	 {De Gruyter Expositions in Mathematics},
  VOLUME =	 4,
  PUBLISHER =	 {Walter de Gruyter \& Co.},
  address =	 {Berlin},
  YEAR =	 1992,
  PAGES =	 {xiv+891},
  ISBN10 =	 {3-11-012892-6},
  MRCLASS =	 {20D10 (20F17)},
  MRNUMBER =	 1169099,
  MRREVIEWER =	 {R. Bryce},
  DOI =		 {10.1515/9783110870138},
  isbn =	 9783110870138,
  URL =		 {https://doi.org/10.1515/9783110870138},
}

@Article{GuarnieriVendramin17,
  author =	 {L. Guarnieri and L. Vendramin},
  title =	 {Skew-braces and the {Y}ang-{B}axter equation},
  journal =	 {Math.\ Comp.},
  year =	 2017,
  volume =	 86,
  number =	 307,
  month =	 mar,
  pages =	 {2519--2534},
  MRCLASS =	 {16T25 (81R50)},
  MRNUMBER =	 3647970,
  doi =		 {10.1090/mcom/3161}
}

@book{Huppert67,
  address =	 berlin,
  author =	 {B. Huppert},
  publisher =	 springer,
  series =	 grund,
  title =	 {{E}ndliche {G}ruppen {I}},
  volume =	 134,
  year =	 1967,
  isbn =	 {978-3-642-64982-0},
  pages =	 {xii+793},
  MRCLASS =	 {20.25},
  mrnumber =	 {MR0224703 (37 {\#}302)},
  doi =		 {10.1007/978-3-642-64981-3}
}

@Book{Kurosh74,
  author =	 {A. G. Kurosh},
  title =	 {General algebra. Lectures of the 1969--1970 academic
                  year},
  publisher =	 {Moscow University Publishing House},
  year =	 1974,
  address =	 {Moscow},
  note =	 {In Russian}
}

@Article{LetourmyVendramin24,
  author =	 {T. Letourmy and L. Vendramin},
  title =	 {{S}chur covers of skew braces},
  journal =	 jalg,
  volume =	 644,
  year =	 2024,
  month =	 apr,
  pages =	 {609--654},
  doi =		 {10.1016/j.jalgebra.2024.01.021}
}

@article{Rump07,
  author =	 {W. Rump},
  journal =	 {J. Algebra},
  pages =	 {153-170},
  title =	 {Braces, radical rings, and the quantum
                  {Y}ang-{B}axter equation},
  volume =	 307,
  doi =		 {10.1016/j.jalgebra.2006.03.040},
  MRCLASS =	 {16Y99 (16W30)},
  MRNUMBER =	 2278047,
  url =		 {https://doi.org/10.1016/j.jalgebra.2006.03.040},
  year =	 2007,
  month =	 jan
}

@InProceedings{Sysak11-Ischia10,
  author =	 {Ya. P. Sysak},
  title =	 {The adjoint group of radical rings and related
                  questions},
  booktitle =	 {Ischia Group Theory 2010.  Proceedings of the
                  Conference, Ischia, Naples, Italy, 14 – 17 April
                  2010},
  year =	 2011,
  editor =	 {M. Bianchi and P. Longobardi and M. Maj and
                  C. M. Scoppola},
  pages =	 {344--365},
  month =	 sep,
  address =	 {Singapore},
  publisher =	 {World Scientific},
  doi =		 {10.1142/9789814350051_0027}
}

@Article{Vendramin19-agta,
  author =	 {L. Vendramin},
  title =	 {Problems on skew left braces},
  journal =	 {Adv. Group Theory Appl.},
  year =	 2019,
  volume =	 7,
  pages =	 {15--37},
  doi =		 {10.32037/agta-2019-003}
}

@Article{Yang67,
  author =	 {C. N. Yang},
  title =	 {Some exact results for many-body problem in one
                  dimension with repulsive delta-function interaction},
  journal =	 {Phys. Rev. Lett},
  year =	 1967,
  volume =	 19,
  issue =	 23,
  month =	 dec,
  pages =	 {1312--1315},
  MRCLASS =	 {81.20},
  MRNUMBER =	 261870,
  doi =		 {10.1103/PhysRevLett.19.1312},
  url =		 {https://doi.org/10.1103/PhysRevLett.19.1312}
}

@Article{Zhu22,
  author =	 {H. Zhu},
  title =	 {The construction of braided tensor categories from
                  {H}opf braces},
  journal =	 {Linear Multilinear Algebra},
  year =	 2022,
  volume =	 70,
  number =	 16,
  pages =	 {3171--3188},
  doi =		 {10.1080/03081087.2020.1828249}
}
\end{document}